\documentclass[12pt]{amsart}
\usepackage{amsmath,amscd,amsthm,amsfonts, amssymb,amsxtra}
\usepackage[all]{xy}
\SelectTips{cm}{}
\usepackage{hyperref}
  \hypersetup{colorlinks=true,citecolor=blue}
 \usepackage[font=scriptsize]{caption}

\CDat
\newtheorem{thm}{Theorem}[section]  
\newtheorem*{un-no-thm}{Theorem}
\newtheorem{cor}[thm]{Corollary}     % Numbered along with thm
\newtheorem{lem}[thm]{Lemma}         % Numbered along with thm
\newtheorem{prop}[thm]{Proposition}

\newtheorem{bigthm}{Theorem}
\newtheorem{bigcor}[bigthm]{Corollary}

\newtheorem{bigadd}[bigthm]{Addendum}

\theoremstyle{definition}
\newtheorem{defn}[thm]{Definition}   % Numbered along with thm

\theoremstyle{definition}
\theoremstyle{definition}

\theoremstyle{remark}
\newtheorem{rem}[thm]{Remark}

\newtheorem*{acks}{Acknowledgements}
 
\newtheorem*{out}{Outline}

\begin{document}
\title{On the moduli space of $A_\infty$-structures}
\date{\today}
\author{John R. Klein}
\address{Mathematics Department, Wayne State University, Detroit, MI USA 48202}
\email{klein@math.wayne.edu}
%\author{Andrew Salch}
%\address{Wayne State University, Detroit, MI 48202}
%\email{asalch@math.wayne.edu}
\author{Sean Tilson}
\address{Institut f\"ur Mathematik, Universit\"at Osnabr\"uck, 49076 Osnabr\"uck, Germany}
\email{sean.tilson@uni-osnabrueck.de}

\begin{abstract} We study the moduli space of $A_\infty$ structures 
on a topological space as well as the moduli space of $A_\infty$-ring
structures on a fixed module spectrum. In each case we show that the moduli
space sits in a homotopy fiber sequence in which the other terms 
are representing spaces for Hochschild cohomology.
\end{abstract}

\dedicatory{To Tom Goodwillie on his sixtieth birthday.}
\thanks{The first author was partially supported by the NSF}
\maketitle%\pagestyle{fancy}
\setlength{\parindent}{15pt}
\setlength{\parskip}{1pt plus 0pt minus 1pt}
\def\Top{{\rm Top}}
\def\wTop{\text{\rm w}\bold T}
\def\wT{\text{\rm w}\bold T}
\def\Sp{\bold S\bold p}
\def\vo{\varOmega}
\def\vs{\varSigma}
\def\smsh{\wedge}
\def\flush{\flushpar}
\def\id{\text{id}}
\def\dbslash{/\!\! /}
\def\codim{\text{\rm codim\,}}
\def\:{\colon}
\def\holim{\text{\rm holim\,}}
\def\hocolim{\text{\rm hocolim\,}}
\def\colim{\text{\rm colim\,}}
\def\hodim{\text{\rm hodim\,}}
\def\hocodim{\text{hocodim\,}}
\def\Bbb{\mathbb}
\def\bold{\mathbf}
\def\Aut{\text{\rm aut}}
\def\cal{\mathcal}
\def\Sec{\text{\rm sec}}
\def\Secst{\text{\rm sec}^{\text{\rm st}}}
\def\maps{\text{\rm map}}
\def\orb{\cal O}
\def\hoP{\text{\rm ho}P}
\def\endo{\text{\rm end}}
\def\ad{\text{\rm ad}}
\newcommand{\HH}{\operatorname{HH}}
\newcommand{\bimod}{\operatorname{bimod}}
\newcommand{\Der}{\operatorname{Der}}
\newcommand{\End}{\operatorname{end}}
\def\Mod{\text{\rm mod}}
\def\Alg{\text{\rm alg}}
\def\Mon{\text{\rm Mon}}

\setcounter{tocdepth}{1}
\tableofcontents
\addcontentsline{file}{sec_unit}{entry}
%\endtableofcontents

\section{Introduction \label{intro}}

There is a long history in algebraic topology of studying homotopy invariant versions of classical algebraic
structures. In the 1960s, a theory of $A_\infty$-spaces, that is, spaces equipped with
a coherently homotopy associative multiplication, was developed by
Stasheff \cite{Stasheff1}, \cite{Stasheff2} and extended by Boardman and Vogt \cite{Boardman-Vogt}. One of the main results of this  
theory is that such a space has the homotopy type of a loop space provided that its
monoid of path components forms a group. There was a hint in Stasheff's work of an obstruction theory for deciding when a space admits an $A_\infty$-structure. Later work by Robinson
in the context of $A_\infty$-ring spectra developed such an obstruction theory, in which
the obstructions lie in Hochschild cohomology 
\cite{Robinson1}, \cite{Robinson2} (see also \cite{Angeltveit}). Robinson's theory is directed
to the question as to whether the moduli space is non-empty and if so, the
problem of enumerating its path components.
Furthermore, Robinson's theory depends on an explicit model
for the Stasheff associahedron and is therefore not in any obvious way ``coordinate free.'' 
The purpose of this paper is to explain from a different perspective
 why Hochschild cohomology arises when studying moduli problems
associated with $A_\infty$-spaces and $A_\infty$-rings. 

The main problem addressed in this paper is identifying the homotopy type of 
the moduli space of $A_\infty$-ring structures on a fixed spectrum. 
As a warm up, we first investigate the related problem of  
indetifying the homotopy type of the moduli space
of $A_\infty$-structures on a fixed topological space.

\subsection*{$A_\infty$-structures on a space} 
An $A_\infty$-space is a based space which admits a mulitplication
that is associative up to higher homotopy coherence.
According to Boardman and Vogt \cite[th.~1.27]{Boardman-Vogt}, 
it is always possible to rigidify an $A_\infty$-space to a topological monoid
in a functorial way. Moreover, there is an appropriate sense in which the
homotopy category of $A_\infty$-spaces is equivalent to the homotopy
category of topological monoids (even more is true: 
Proposition \ref{prop:A-infty=monoid} below
says that in the derived sense, function complexes of topological
monoids are weak equivalent to the corresponding function complexes
of $A_\infty$-maps; this appears to be well-known \cite{Lurie}).
Hence, rather than working with $A_\infty$-spaces,
we can use topological monoids to define the moduli space of $A_\infty$-structures.

\begin{defn} \label{defn:CsubX}
If $X$ is a connected based space, 
Let $\cal C_X$ be the category whose objects are pairs
\[
(M,h)
\]
in which $M$ is a topological monoid and $h\: X \to M$ is a weak homotopy equivalence
of based spaces. A morphism $(M,h) \to (M',h')$ is a homomorphism $f\: M \to M'$ such that
$h ' = f\circ h$.

We define the {\it moduli space}
\[
\cal M_X = |\cal C_X|
 \]
to be the geometric realization of the nerve of $\cal C_X$.
\end{defn}

For a map 
$f\: Z \to Y$ of (unbased) spaces
we define {\it unstable Hochschild ``cohomology''}
\[
\cal H(Z;Y) 
\]
 to be the space of factorizations \[Z \to LY \to Y\] of $f$ in which 
$LY = \maps(S^1,Y)$  is the free loop space of $Y$ and $LY \to Y$ is the fibration given by evaluating a free
loop at the base point of $S^1$. Notice that $\cal H(Z;Y)$ has a preferred basepoint
given by the factorization $Z \to LY \to Y$ in which $Z \to LY$ is
the map given by sending $z \in Z$ to the constant loop with value $f(z)$.

If we fix $Y$, and let $\Top_{/Y}$ be the category of 
spaces over $Y$, then the assignment $Z\mapsto \cal H(Z;Y)$ defines a contravariant functor
$\Top_{/Y} \to \Top_\ast$.

\begin{rem} To justify the terminology, let $G$ be a topological
group. Consider the case when $Z = BG = Y$ is the classifying space
of $G$.  Then the fibration 
$LBG \to BG$ is fiber homotopy equivalent to the Borel construction
\[
EG \times_G G^\ad \to BG\, .
\]
in which $G^\ad$ denotes a copy of $G$ considered
as a left $G$-space via the adjoint action $g\cdot x := gxg^{-1}$ (see
e.g., \cite[\S9]{KSS}).

If we we make the latter fibration into a fiberwise spectrum 
by applying the suspension spectrum
functor to each fiber, then the associated  spectrum of global
sections is identified with the topological Hochschild cohomology spectrum 
\[
\HH^\bullet(S[G];S[G])\, , 
\] where $S[G] := \Sigma^\infty (G_+)$ is the 
suspension spectrum of $G$ (also known as the group ring of $G$ over the sphere
spectrum). 
%Hence, we can regard $\cal H(BG;BG)$ is a kind of unstable 
%analog of topological Hochschild cohomology. More precisely,
%up to weak equivalence $\cal H(BG;BG)$ is a homotopy fixed point
% space $(G^{\ad})^{hG}$ of $G$ acting on $G^\ad$ whereas $\HH^\bullet(S[G];S[G])$ is
%the homotopy fixed point spectrum $S[G^\ad]^{hG}$ of $G$ acting on $S[G^\ad]$.
\end{rem}

Recall for an $A_\infty$-space $X$, there is an inclusion $\Sigma X \to BX$
which we may regard as a morphism of $\Top_{/BX}$.
In particular, there is a restriction map 
\begin{equation} \label{eqn:map_of_fibration}
\cal H(BX;BX) \to \cal H(\Sigma X;BX)\, .
\end{equation}

\begin{bigthm} \label{bigthm:spaces} Assume $X$ has a preferred $A_\infty$-structure. Then the map \eqref{eqn:map_of_fibration} sits in a homotopy fiber
sequence
\[
\Omega^2 \cal M_X \to \cal H(BX,BX) \to \cal H(\Sigma X,BX)
\]
in which the double loop space $\Omega^2 \cal M_X$ is identified with the homotopy fiber at
the basepoint.
\end{bigthm}

In fact, after some minor modification, 
the homotopy fiber sequence of Theorem \ref{bigthm:spaces} 
admits a double delooping, enabling one to recover the moduli space $\cal M_X$. 
If the underlying space of an object $Z\in \Top_{/Y}$ 
is equipped with a basepoint $\ast_Z \in Z$, then 
$\ast_Z$ becomes an object in its own right and we have a morphism $\ast_Z \to Z$.
This in turn induces a fibration of based spaces
\[
\cal H(Z;Y) \to \cal H(\ast_Z;Y)\, .
\]
We take 
\[
\tilde {\cal H}(Z,Y) =\text{fiber}(\cal H(Z;Y) \to \cal H(\ast_Z;Y))
\] 
to be the fiber at the basepoint. 
Concretely, this is the space of factorizations $Z \to LY \to Y$ in which 
$Z \to LY$ is a based map (where $LY$ is given the basepoint consisting
of the constant loop defined by the image of $\ast_Z$ in $Y$).

We can 
regard $\tilde {\cal H}(Z,Y)$ as a kind of reduced unstable cohomology.
It  defines a contravariant functor on $\Top_{\ast/Y}$, the category of based spaces over $Y$. In particular, we have a map
\begin{equation} \label{eqn:based_map_of_fibration}
\tilde{\cal H}(BX;BX) \to \tilde{\cal H}(\Sigma X;BX)\, .
\end{equation}

\begin{rem} The if $Z' \to Z$ is a map of based spaces over a space $Y$, then 
it is easy to see that the associated commutative diagram of based spaces
\[
\xymatrix{
\tilde{\cal H}(Z;Y)  \ar[r] \ar[d] & \tilde{\cal H}(Z';Y)\ar[d] \\
\cal H(Z;Y)  \ar[r]  & \cal H(Z';Y) \\
}
\]
is homotopy cartesian. In particular since $\Sigma X \to BX$ is a
based map, we can replace
$\cal H$ by $\tilde {\cal H}$ in Theorem \ref{bigthm:spaces}
and the statement of the theorem remains valid. 
\end{rem}

The advantage of using the reduced version is made clear by the following
companion to Theorem \ref{bigthm:spaces}.

\begin{bigadd} \label{bigadd:spaces} The map \eqref{eqn:based_map_of_fibration}
has a preferred non-connective double delooping 
$\cal B^2 \tilde{\cal H}(\Sigma X;X) \to \cal B^2\tilde{\cal H}(BX;X)$ which sits in 
a homotopy fiber sequence
\[
\cal M_X \to \cal B^2 \tilde{\cal H}(BX,BX) \to \cal B^2 \tilde{\cal H}(\Sigma X,BX)\, .
\]
\end{bigadd}

\begin{rem} The main idea of the proof of Addendum \ref{bigadd:spaces} involves
showing that the (moduli) space of $A_\infty$-spaces $Y$ which are weakly equivalent to $X$ 
(in an unspecified way) gives a double delooping of 
$\tilde{\cal H}(BX,BX)$, and similarly the space of based spaces $Y$ which
are weakly equivalent to $X$ is a double delooping of $\tilde{\cal H}(\Sigma X,BX)$.
\end{rem}

\subsection*{$A_\infty$-ring structures} 
In recent years the subject of algebraic structures on spectra
has been profoundly transformed by the existence of improved models
for the category of spectra that admit a strictly associative and commutative smash product.
What was once called an $A_\infty$-ring spectrum can now be regarded as simply
a monoid object in one of the new models. 

Henceforth,
we work in one of the good
symmetric monoidal categories of spectra.
To fix our context
we work in symmetric spectra. We will fix a connective commutative monoid
object $k$ in symmetric spectra. This will function as our ground ring.
The category of (left) $k$-modules will be denoted by $k\text{-}\Mod$.

By virtue of \cite{Schwede-Shipley}
(cf.\ \cite[chap.~II]{EKMM}), the category of $A_\infty$-ring spectra
is modeled by
the category of monoid objects in $k\text{-}\Mod$ and their homomorphisms. 
This category is denoted by $k\text{-}\Alg$. An object of $k\text{-}\Alg$ is called a {\it $k$-algebra}. There is a forgetful functor $k\text{-}\Alg \to k\text{-}\Mod$.

A morphism of either $k\text{-}\Mod$ or 
$k\text{-}\Alg$ is a fibration or a weak equivalence if it is one when considered in the underlying category of symmetric spectra. A morphism is a cofibration if it satisfies the
left lifting property with respect to the acyclic fibrations.
According to \cite[thm~4.1]{Schwede-Shipley},
these notions underly a model structure on $k\text{-}\Mod$ and $k\text{-}\Alg$.

\begin{defn} For a $k$-module $E$, the {\it moduli space}
\[
\cal M_E
\]
is the classifying space of the category ${\cal C}_E$ whose objects are pairs
\[
(R,h)
\]
in which $R$ is a $k$-algebra and
 $h\: E \to R$ is a weak equivalence of  
$k$-modules. 
A morphism $(R,h) \to (R',h')$ is a map of $k$-algebras $f\: R \to R'$ such that
$h' = f\circ h$. 
\end{defn}

\begin{rem} Using the same notation for the moduli space in each of the settings,
i.e., $k$-modules and spaces,  should not give rise to any confusion, since the subscript
clarifies the context.
\end{rem}

Fix a $k$-algebra $R$ and consider the $k$-algebra 
$R^e := R\smsh_k R^{\text{op}}$. A (left) $R^e$-module
is also known as an {\it $R$-bimodule}. In particular, $R$ is an $R$-bimodule.

%\begin{defn}[cf.\ Dwyer-Kan \cite{DK1}]
%If ${\cal C}$ is a model category, we denote by 
% ${\cal C}(X,Y)$ the homotopy function complex of maps from $X\to Y$.
 %\end{defn}

\begin{defn} The {\it topological Hochschild cohomology} of $R$ with coefficients in an
$R$-bimodule $M$ is the homotopy function complex
\[
 \HH^\bullet(R;M) := {R^e\text{-}\Mod}(R,M)\, ,
\]
given by the hammock localization of $R$-bimodule maps from $R$
to $M$ appearing in \cite[3.1]{DK1}.\footnote{Alternatively, one can define $\cal \HH^\bullet(R;M)$ as 
derived simplicial hom, that is 
the space of maps $R^c \to M^f$ in which $R^c$ is a cofibrant approximation of $R$
and $M^f$ is a fibrant approximation of $M$ in $R^e$-modules; 
here we are using a simplicial model structure
on $R^e$-modules to define mapping spaces.} 
\end{defn}

Suppose now that $1\: R \to A$ is a homomorphism of $k$-algebras. Then
$A$ is an $R$-bimodule, and $1\: R \to A$ is a map of $R$-bimodules.
In this way $\HH^\bullet(R;A)$ inherits a distinguished basepoint.
Furthermore, if $f\: R'\to R$ is a $k$-algebra map, then the induced
map $f^\ast\: \HH^\bullet(R;A) \to \HH^\bullet(R';A)$ is a map of based spaces.
In particular, the homomorphism $R^e \to R$ induces a map $\HH^\bullet(R;A) \to \HH^\bullet(R^e;A) \simeq \Omega^\infty A$. Let $m\: R^e \to R$ be the multiplication.
Then the composite $1\circ m\: R^e\to A$ is an $R^e$-module map so 
there is a distinguished basepoint of $\HH^\bullet(R^e;A)$ which corresponds to the unit of 
$\Omega^\infty A$.

\begin{defn} For a $k$-algebra homomorphism $1\: R \to A$, the {\it reduced topological Hochschild cohomology} is defined to be the homotopy fiber
\[
D^\bullet(R;A) := \text{hofiber}( \HH^\bullet(R;A) \to  \HH^\bullet(R^e;A))
\]
taken at the distinguished basepoint. 
 \end{defn}

Let 
\[
T\: k\text{-}\Mod \to k\text{-}\Alg
\] 
be the {\it free functor} (i.e., the tensor algebra); this is left adjoint to the forgetful
functor $k\text{-}\Alg \to k\text{-}\Mod$, so if $R\in k\text{-}\Alg$ is an object, we have a $k$-algebra
map $TR \to R$. In particular, $R$ has the structure of a $TR$-bimodule. 

\begin{bigthm}\label{bigthm:main} Assume
$R \in k\text{-}\Mod$ is equipped with the structure of an $k$-algebra.  
Then the map
$
D^\bullet(R;R) \to 
D^\bullet(TR;R)
$
has a preferred (non-connective) double delooping 
 $\cal B^2D^\bullet(R;R) \to 
\cal B^2D^\bullet(TR;R)$
which sits in a homotopy fiber sequence of based spaces
\[
\cal M_R \to \cal B^2D^\bullet(R;R) \to 
\cal B^2D^\bullet(TR;R)\, .
\]
\end{bigthm}

\begin{rem} As {\it unbased} spaces, $\cal B^2D^\bullet(R;R)$
and $\cal B^2D^\bullet(TR;R)$ 
are constructed so as to depend only on the underlying $k$-module structure of $R$. 
However, the $k$-algebra structure
induces a preferred basepoint making $\cal B^2D^\bullet(R;R) \to 
\cal B^2D^\bullet(TR;R)$ into a based map.

The double delooping of $D^\bullet(R;R)$ is formally rigged 
so that its set of path components  gives us the
correct value of $\pi_0(\cal M_R)$. Hence, one cannot use
 Theorem \ref{bigthm:main} to compute $\pi_0(\cal M_R)$. 
It seems that the only sufficiently 
 general approach to computing path components is the obstruction
theory of  \cite{Robinson1}, \cite{Robinson2}, \cite{Angeltveit},
 which is tailored to making such computations.
\end{rem}

By taking the two-fold loop spaces and noticing that
the reduced cohomology spaces are obtained by fibering 
the corresponding unreduced cohomology over $\Omega^\infty R$ in each case, we
infer

\begin{bigcor} \label{cor:main}
There is a homotopy fiber sequence
\[
\Omega^2 \cal M_R \to {\HH}^\bullet(R;R) \to 
{\HH}^\bullet (TR;R)\, ,
\]
in which $\Omega^2\cal M_R$ is identified with homotopy fiber
at the basepoint of ${\HH}^\bullet (TR;R)$ that is
 associated with the $TR$-bimodule map $TR \to R$.
 \end{bigcor}

\begin{rem} Corollary \ref{cor:main} is reminiscent of Lazarev's \cite[thm.~9.2]{Lazarev3}. However, the moduli space appearing there is different from ours:  
the points of Lazarev's moduli space are
$k$-algebras $R'$ whose weak homotopy type as a $k$-module is $R$, but $R'$ does not come equipped with a choice of $k$-module equivalence to $R$.
\end{rem}

The proof of Theorem \ref{bigthm:main} uses various identifications
of $D^\bullet(R;A)$ in the case when $A$ is the $R$-bimodule arising
from a $k$-algebra homomorphism $1\:R\to A$. One of the key identifications 
interprets $D^\bullet(R;A)$ as the loop space of the space of 
derived algebra maps $R\to A$:

\begin{bigthm} \label{bigthm:crux} 
There is a natural weak equivalence
\[
D^\bullet(R;A)\simeq \Omega_1 {k\text{-}\Alg}(R,A)  \, .
\]
where the right side is the based loop space at $1$ of the homotopy function
complex of $k$-algebra homomorphisms from $R$ to $A$.
\end{bigthm}

\begin{out} Section \ref{sec:prelim} is mostly
 language. In section \ref{sec:spaces}
we give the proof of Theorem \ref{bigthm:spaces} and Addendum \ref{bigadd:spaces}.
This section is independent of the rest of the paper, and we view it as
motivation for the $A_\infty$-ring case. Section \ref{sec:crux} is the meat
of the paper. The hardest part is to establish an augmented 
equivalence of $A_\infty$-rings between
the trivial square zero extension $S \vee S^{-1}$ and the Spanier-Whitehead dual
of $S^1_+$ (this appears in Proposition \ref{prop:free-loop}). In section \ref{sec:main}
we deduce Theorem \ref{bigthm:main}. Section \ref{sec:augment} develops an augmented version of
Theorem \ref{bigthm:main}. In section \ref{sec:examples} we study the homotopy
type of the moduli space
in two examples: the trivial square zero extension $S \vee S^{-1}$ and the commutative 
group ring case $S[G]$.
\end{out}

\begin{acks} The authors thank Dan Dugger for discussions connected with
\S\ref{sec:crux}. We are especially indebted to 
Bill Dwyer and Mike Mandell for helping us with the proof
of Proposition \ref{prop:free-loop}. We also learned a good deal from discussions with Andrew Salch.
\end{acks}

\section{Preliminaries \label{sec:prelim}}

\subsection*{Categories}
Let $\cal C$ be a category. If $X,Y \in \cal C$ are objects, we let 
$\hom_{\cal C}(X,Y)$ denote the set of arrows from $X$ to $Y$.

We let $|\cal C|$ be the (geometric) realization of (the nerve of) $\cal C$, i.e., the 
classifying space. 
Many of the categories $\cal C$ considered in this paper are not small. As usual, in order to avoid set theoretic difficulties and have a well-defined homotopy type $|\cal C|$, one has to make certain modifications, such as
working in a Grothendieck universe. We will implicitly assume this has been done.

We say that a functor $f\: \cal C \to \cal D$ is a {\it weak equivalence}
if it induces a homotopy equivalence on realizations. We say that 
a composition of functors
\[
\cal C @> f >> \cal D @> g >> \cal E
\]
forms a homotopy fiber sequence if after realization if there is
a preferred choice of based, contractible space
$U$ together with a commutative diagram
\[
\xymatrix{
|\cal C| \ar[r]^{|f|} \ar[d] & |\cal D| \ar[d]^{|g|} \\
U \ar[r] & |\cal E|}
\]
which is homotopy cartesian. In this paper, $U = |\cal U|$ for
a suitable pointed category $\cal U$, and the diagram arises from a commutative
diagram of functors. We call $\cal U$ the {\it contracting data.}

\subsection*{Model categories}

The language of model categories will be used throughout the paper.
If $\cal C$ is a model
category, we let 
\[
{\cal C}(X, Y)
\]
denote the {\it homotopy function complex} from $X$ to $Y$, where we use the specific model given by the {\it hammock localization} of Dwyer and Kan \cite[3.1]{DK1}. In particular,
any zig-zag
of the form
\[
X = X_0 @< \sim << X_1 @>>> X_2 @< \sim << \cdots @>>> X_{n-1} @< \sim << X_n = Y
\]
represents a point of ${\cal C}(X, Y)$.
If $\cal C$ is a simplicial model
category (in the sense of Quillen \cite[chap.~2,\S1]{Quillen}), $X$ is cofibrant and $Y$ is fibrant, then ${\cal C}(X, Y)$
has the homotopy type of the simplicial function space $F_{\cal C}(X,Y)$.

Let ${\cal C}$ be a model category. If  $X\in \cal C$ is an object, let $h{\cal C}(X)$ 
denote the category
consisting of all objects weakly equivalent to $X$, in which a morphism
is a weak equivalence of ${\cal C}$.
An important result used in this paper, due to Dwyer and Kan \cite[prop.~2.3]{DK2} is
the weak equivalence of based spaces 
\begin{equation} \label{eqn:DK-equivalence}
|h\cal C(X)| \simeq B h \Aut(X)\, ,
\end{equation}
where  $h \Aut(X)$ is the simplicial monoid  of homotopy automorphisms of $X$, 
which is the union of those components of ${h\cal C}(X,X)$
that are invertible in the monoid $\pi_0({h\cal C}(X,X))$.

As mentioned above,
there are simplicial model category structures on 
$k\text{-}\Mod$ and $k$-$\Alg$  (see Schwede and Shipley \cite{Schwede-Shipley}). 
In each case the fibrations
and weak equivalences are determined by the forgetful functor to symmetric spectra,
and the cofibrations are defined by the left lifting property with respect to the acyclic fibrations.

\subsection*{Spaces} Let $\Top$ be the category of compactly generated weak
Hausdorff spaces. When taking products, we always mean in the compactly generated sense.
 Function spaces are to be given the compactly generated, compact-open topology. 
 There is a well-known simplicial model category structure
 on $\Top$ in which a fibration is a Serre fibration,
 a weak equivalence is a weak homotopy equivalence and a cofibration is defined by
 the left lifting property with respect to the acyclic fibrations. Similarly,
 $\Top_\ast$, the category of based spaces, is a model category where the
 weak equivalences, fibrations and cofibrations are given by applying the forgetful
 functor to $\Top$.

\section{Proof of Theorem \ref{bigthm:spaces} and Addendum \ref{bigadd:spaces}
\label{sec:spaces}}

\subsection*{Interpretation of the moduli space}
Let $X$ be a cofibrant based space.
Let $JX$ be the free monoid on the points of $X$. Then $JX$ is the topological
monoid object of $\Top$ given by
\[
X \cup X^{\times 2} \cup \cdots \cup X^{\times n} \cup \cdots 
\]
where a point in $X^{\times n}$ represents a word of length $n$.
Multiplication is defined by word amalgamation and 
the identifications are given by reducing words, i.e., dropping the basepoint
of $X$
whenever it appears.

The moduli space $\cal C_X$ has an alternative definition as the category whose objects are
pairs 
\[
(M,h)
\]
in which $h\: JX \to M$ is a (topological) monoid homomorphism which restricts
to a weak homotopy equivalence $X \to M$. A map $(M,h) \to (M',h')$
is a homomorphism $f\: M \to M'$ such that $h' = f\circ h$.
There is then a decomposition
\[
\cal C_X = \coprod_{[N,h]} C_X(N,h)
\]
where $[N,h]$ runs through the components of $\cal C_X$, and
$C_X(N,h)$ denotes the component of $\cal C_X$ given by $[N,h]$.

Let $\Mon$ be the category of topological monoid
objects of $\Top_\ast$. Then $\Mon$ inherits the structure of a simplicial
model category in which the weak equivalences and the fibrations are 
defined by the forgetful functor $\Mon \to \Top_\ast$, and  cofibrations are defined by the left lifting property with respect to the acyclic fibrations (this follows
from \cite{Schwede-Shipley}, as well as
 \cite{Schwaenzl-Vogt}).
Cofibrant objects are retracts of those
objects which are built up from the trivial monoid by sequentially attaching 
free objects, where a free object is of the form $JY$. Every object is fibrant.

If $M \in \Mon$ is an object, then we can form the under category
$M\backslash\Mon$ whose objects are pairs $(N,h)$ in which $h\: M \to N$ is a monoid map. 
A morphism $(N,h) \to (N',h')$ is a monoid map $f\: N\to N'$ such that $h' = fh$.
By Quillen \cite[chap.~2, prop.~6] {Quillen}  
$M\backslash\Mon$ forms a simplicial model category in which a fibration, cofibration
and weak equivalence are defined by the forgetful functor. When there is no confusion
 we simplify 
the notation and
 drop the structure map when referring to an object: $N$ henceforth refers to $(N,h)$. 

Let $h\Mon \subset \Mon$ be the category of weak equivalences and let
$N \in M\backslash \Mon$ be a cofibrant object, where $M\backslash \Mon$
is the (comma) category of objects of $\Mon$ equipped map from $M$ . By the 
Dwyer-Kan equivalence \eqref{eqn:DK-equivalence}, there is a homotopy equivalence
\[
|M\backslash h\Mon(N)| \simeq Bh\Aut(N\text{ rel } M)\,  ,
\]
where the right side is the simplicial monoid of homotopy automorphisms
of $N$ relative to $M$. 

We now specialize to the case where $M = JX$,
$N$ is cofibrant, and the composite $X \to JX \to N$ is a weak equivalence of based spaces.
We claim that there is a homotopy fiber sequence of based spaces
\begin{equation} \label{eqn:fiber-sequence}
Bh\Aut(N\text{rel } JX) \to Bh\Aut_{\Mon}(N) \to Bh\Aut_{\Top_\ast}(N)\, ,
\end{equation}
where each map is a forgetful map. The claim can
be established as follows: 
consider the forgetful maps of function spaces
\[
F_{\Mon}(N,N;\text{ rel } JX) \to F_{\Mon}(N,N) \to F_{\Top_\ast}(N,N) \, ,
\]
These maps form a homotopy fiber sequence of topological monoids (here we use the fact that 
$F_{\Top_\ast}(N,N) \simeq F_{\Mon}(JX,N)$). Taking homotopy invertible components
yields a  homotopy fiber sequence of homotopy automorphisms. The fiber sequence
\eqref{eqn:fiber-sequence} is then obtained by taking classifying spaces.

We may enhance the homotopy fiber sequence \eqref{eqn:fiber-sequence} 
to another one as follows:
\begin{equation} \label{eqn:enhanced}
\coprod_{[N,h]} Bh\Aut(N\text{rel } JX)  \to 
\coprod_{[N]} Bh\Aut_{\Mon}(N) 
\to Bh\Aut_{\Top_\ast}(X)\, .
\end{equation}
Here, the middle term is a disjoint union over the components of 
$h\Mon$ whose objects have underlying space weakly equivalent to $X$, and
the first term is the disjoint union over the components of $\cal C_X$.
Furthermore, the Dwyer-Kan equivalence 
\eqref{eqn:DK-equivalence} gives an identification
\[
|\cal C_X(N,h)|    \simeq  Bh\Aut(N\text{ rel } JX)\, .
\]
Consequently, the homotopy fiber sequence
\eqref{eqn:enhanced} arises from the homotopy fiber sequence of categories
\begin{equation} \label{eqn:enhanced-even-more}
\cal C_X  \to  \coprod_{[N]} h\Mon(N)  \to h\Top_\ast(X)\, ,
\end{equation}
where the middle term is
a coproduct indexed over the components of $\Mon$ which
have the weak homotopy type of $X$. The 
functors appearing in \eqref{eqn:enhanced-even-more} are  the forgetful functors.
In \eqref{eqn:enhanced-even-more}, we can take the contracting data $\cal U$ 
to be the  category whose objects are pairs $(Y,h)$ in which $Y$ is a based space
and $h\: X \to Y$ is a weak equivalence, where a morphism $(Y,h) \to (Y',h')$ is a
map $f\: Y \to Y'$ such that $h' = f\circ h$. Clearly,
$(X,\text{id})$ is an initial object so $\cal U$ is contractible. The functor
$\cal U \to h\Top_\ast(X)$ is the forgetful functor defined by $(Y,h) \mapsto Y$, and the functor $\cal C_X \to \cal U$ is the forgetful functor defined by the inclusion.

\begin{lem} \label{lem:hochschild}  Assume $X$ is a connected cofibrant based space which is
equipped the structure of a topological monoid. Then there are  weak equivalences
\[
\Omega^2 |h\Top(X)| \,\,  \simeq \,\, \tilde{\cal H}(\Sigma X,BX) \, .
\]
and
\[
\Omega^2 |h\Mon(X)| \,\,  \simeq \,\, \tilde{\cal H}(BX,BX) \, .
\]
\end{lem}

\begin{proof} Since $|h\Top_\ast(X)| \simeq Bh\Aut_\ast(X)$ it suffices
to identify $\Omega^2Bh\Aut_\ast(X)$   with $\tilde{\cal H}(\Sigma X;BX)$.
Since $h\Aut_\ast(X)$ is group-like, we have
\[
\Omega^2Bh\Aut_\ast(X) \simeq \Omega_1 h\Aut_\ast(X) = \Omega_1 F_\ast(X,X)\, ,
\]
where $\Omega_1$ denotes loops taken at the identity, and $F(X,X)$ is the
function space of based maps self-maps of $X$. Hence,
\begin{align*}
\Omega_1 F_\ast(X,X)& \,\,  \simeq \,\,  \Omega_1 F_\ast(X,\Omega BX)\, ,\\
& \,\, \cong\,\, 
\Omega_1 F_\ast(\Sigma X, BX)\, ,\\
& \,\, \cong\,\, 
\tilde{\cal H}(\Sigma X,BX) \, ,
\end{align*}
giving the first part of the lemma.
To prove the second part, we require
\medskip

\noindent{\it Claim:}
For group-like topological monoids $X$ and $Y$, 
the classifying space functor induces a weak equivalence
of homotopy function complexes
\[
\Mon(X,Y) \simeq \Top_\ast(BX,BY)\, .
\]
\medskip

The claim can be proved using model category ideas,
using the Moore loop functor. For the sake of completeness, we sketch an alternative
low-tech argument here.
To prove the claim, it is enough to check the statement
when $X$ is cofibrant (in this instance, $\Mon(X,Y) \simeq F_{\Mon}(X,Y)$).  
It is not difficult to show that such an $X$ is a retract of an object
built up from a point by attaching free objects, where a free object is of the 
form $JU$, in which $U$ is a based space and $JU$ is the free monoid
on the points of $U$. By naturality, it is enough to check the statement when
$X$ itself is inductively built up by attaching free objects. One can now argue
by induction. The basis step is for the zero object $X = \ast$. In this
case the claim is trivial.

An auxiliary step is to check the claim for 
a free object $X = JU$.
Since $J$ is a left adjoint to the forgetful functor $\Mon \to \Top_\ast$,
the function space
of monoid maps $F_{\Mon}(X,Y)$ coincides with the function space
of based maps $F_\ast(U,Y)$. Since $Y$ is group-like,
we have 
\[
F_\ast(U,Y) \simeq F_\ast(U,\Omega BY) \cong F_\ast(\Sigma U,Y) \simeq F_\ast(BJU,Y)\, ,
\]
where we have used a theorem of James to identify $BJU$ with $\Sigma U$.
This concludes the auxiliary step. 

For the inductive step,
suppose $(D,S)$ is a cofibration pair of based spaces and
the claim is true for $X_0$ and let $JS \to X_0$ be a
monoid map. Let $X_1 = \colim(X_0\leftarrow JS \to JD)$.
Then, by 
the fact that (i) function spaces out of pushouts give rise
to pullbacks, and (ii) the classfying space functor preserves homotopy
pushouts, we infer that the claim is true for $X_1$.
This completes the proof sketch of the claim.

The claim implies $|h\Mon(X)| \simeq Bh\Aut_\ast(BX,BX)$, and
we infer
\begin{align*}
\Omega^2 |h\Mon(X)| &\,\,  \simeq \,\, \Omega^2 Bh\Aut_\ast(BX,BX)\, , \\
& \,\, \simeq \,\, \Omega_1 h\Aut_\ast(BX,BX)\, , \\
& \,\, \cong  \,\, \tilde{\cal H}(BX,BX)\, . \qedhere
\end{align*}
\end{proof}

\begin{proof}[Proof of Theorem \ref{bigthm:spaces} and Addendum \ref{bigadd:spaces}]
Using Lemma  \ref{lem:hochschild} and \eqref{eqn:enhanced-even-more}, there is 
a homotopy fiber sequence
\[
\Omega^2 \cal M_X \to \tilde{\cal H}(BX;BX) \to \tilde{\cal H}(\Sigma X;BX)\, .
\]
In the second and third terms, we can replace $\tilde{\cal H}$ by $\cal H$,
since in each case we are fibering over the same space $\cal H(\ast,BX) \simeq X$.

Lemma \ref{lem:hochschild} also shows that the realization of $\coprod_{[N]} h\Mon(N)$ defines a non-connective
double delooping of $\tilde{\cal H}(BX;BX)$ and $h\Top_\ast(X)$ is a (connective) double
delooping of $\tilde{\cal H}(\Sigma X;BX)$. The homotopy fiber sequence 
\eqref{eqn:enhanced-even-more} completes the proof.
\end{proof}

We end this section with a result about the relation between $A_\infty$-spaces
and topological monoids which gives further justification as to why our moduli space $\cal M_X$
really is a description of the moduli space of $A_\infty$-structures on $X$.
According to \cite{Schwaenzl-Vogt}, the category of $A_\infty$-spaces,
denoted here by $\Mon_{A_\infty}$, forms a simplicial
model category where a weak equivalence and a fibration are defined by the forgetful functor
to $\Top_\ast$, and cofibrations are defined by the left lifting property with respect
to the acyclic fibrations.

\begin{prop} \label{prop:A-infty=monoid}
Let $X$ and $Y$ be topological monoids, where $Y$ is group-like. Then the inclusion
of topological monoids into $A_\infty$ spaces induces
a weak equivalence of homotopy function complexes
\[
\Mon(X,Y) \,\, \simeq \,\, \Mon_{A_\infty}(X,Y) \, .
\]
\end{prop}

\begin{proof}   Consider the composition
\[
\Mon(X,Y) \to \Mon_{A_\infty}(X,Y) \to \Top_\ast(BX,BY)\, .
\]
According to \cite[7.7]{Fuchs}, \cite[prop.~1.6]{Boardman-Vogt}
the second map is a weak equivalence. By the claim appearing in 
the proof of Lemma \ref{lem:hochschild}, the composition is a weak equivalence. It follows that the first map is a weak equivalence.
\end{proof}

\begin{rem} Dylan Wilson pointed out to us that
 Proposition \ref{prop:A-infty=monoid} 
 is still true without the group-like condition on $Y$. See 
\cite[thm.~4.1.4.4, prop.~4.1.2.6]{Lurie}.
\end{rem}

\section{Proof of Theorem \ref{bigthm:crux} \label{sec:crux}}

\subsection*{Universal differentials and Derivations}
Following Lazarev, we define the $R$-bimodule of {\it universal differentials}
\[
\Omega_{k\to R}
\]
to be the homotopy fiber of the multiplication map 
\[
R \smsh_S R^{\text{op}} \to R
\]
in the model category $R^e\text{-}\Mod$.
From now on we will assume that $R$ is both fibrant and cofibrant.

For an $R$-bimodule $M$, we consider the trivial square zero extension 
$R \vee M$, which is a $k$-algebra. 
Projection onto the first summand is a morphism of $k$-algebras
$R \vee M \to R$ .
The category of $k$-algebra's over $R$, denoted $k\text{-}\Alg/R$, is a  model
category and $R \vee M$ is then an object of it.

The homotopy function complex
\[
\Der(R,M) := {k\text{-}\Alg/R}(R,R \vee M)
\]
is then defined. We make this into a based space using the inclusion $R\to R\vee M$.

\begin{rem}\label{rem:extension} Suppose that $1\: R \to A$ is a homomorphism of $k$-algebras and $M$ is 
an $A$-bimodule.  Then the diagram
\[
\xymatrix{
R\vee M \ar[r] \ar[d] & A \vee M\ar[d] \\
R \ar[r] & A
}
\]
is homotopy cartesian. We infer that there is a weak equivalence
\[
\Der(R;M) \simeq {k\text{-}\Alg/A}(R,A \vee M)\, .
\]
\end{rem}

\begin{lem} \label{lem:desuspend} There is a weak equivalence 
\[
\Der(R;\Sigma^{-1} M) \simeq \Omega\Der(R;M)\, .
\]
\end{lem}

\begin{proof} The functor $M \mapsto \Der(R;\Sigma^{-1} M)$ preserves homotopy cartesian
squares of bimodules. There is a homotopy cartesian square
\[
\xymatrix{
\Sigma^{-1} M \ar[r] \ar[d] & \ast \ar[d]\\
\ast \ar[r] &  M\, .
}
\]
Furthermore, it is evident that $\Der(R;\ast) = {k\text{-}\Alg/R}(R,R)$ is contractible. Hence, the map
\[
\Der(R,\Sigma^{-1} M) \to \holim(\ast \to \Der(R,M) \leftarrow \ast) \simeq \Omega \Der(R,M)
\]
is a weak equivalence.
\end{proof}

\begin{prop}[Lazarev \cite{Lazarev1}, Dugger-Shipley \cite{Dugger-Shipley}]  \label{prop:Lazarev} For any $R$-bimodule $M$, there is a weak equivalence 
\[
\Der(R,M) \simeq  {R^e\text{-}\Mod}(\Omega_{k\to R},M)\, .
\]
\end{prop}

\begin{rem} Lazarev's proof of this statement contains serious gaps. The 
proof was corrected by Dugger and Shipley, based on an unpublished result of Mandell
(cf.\ \cite[rem.~8.7]{Dugger-Shipley}).
\end{rem}

Given a $k$-algebra homomorphism $1\: R \to A$, we can regard $A$ as an $R$-bimodule
in the evident way. 

\begin{cor}\label{cor:special-case-of-lazarev} There is a weak equivalence
\[
\Der(R,A) \simeq  {R^e\text{-}\Mod}(\Omega_{k\to R},A)\, .
\]
\end{cor}

Consider the fibration sequence
\[
\Omega_{k\to R} \to R \smsh_S R^{\text{op}} \to R \, .
\]
Apply ${R^e\text{-}\Mod}({-},A)$ to this sequence to get a homotopy fiber sequence 
\[
\HH^\bullet(R;A)\to \Omega^\infty A \to {R^e\text{-}\Mod}(\Omega_{k\to R},A)
\]
where we have identified the middle term with 
${R^e\text{-}\Mod}(R^e,A)$.
By shifting the homotopy fiber sequence once over to the left (using the unit component
of $\Omega^\infty A \simeq \HH^\bullet(R^e;A)$ as basepoint),
 we see that there's a weak equivalence
$$
\Omega{R^e\text{-}\Mod}(\Omega_{k\to R},A) \simeq D^\bullet(R;A)\, .
$$
If we combine this with Corollary \ref{cor:special-case-of-lazarev} and Lemma \ref{lem:desuspend}, we infer

\begin{cor} \label{cor:reduced-hh-derivations} There is a weak equivalence 
\[
D^\bullet(R;A) \simeq \Omega \Der(R,A)\, .
\]
\end{cor}

Let $A_1$ be the $k$-algebra
\[
A \vee \Sigma^{-1} A
\]
in which $\Sigma^{-1} A$ is given the evident $A$-bimodule structure making $A_1$
into a trivial square zero extension of $A$.
Also, let $\cal LA$ denote  the $k$-algebra given by the function spectrum
\[
F(S^1_+,A) \, ,
\]
taken in the category of symmetric spectra.
The multiplicative structure on $\cal LA$ arises from the 
multiplication on $A$ and the diagonal map $S^1 \to S^1 \times S^1$.
Then we have

\begin{prop} \label{prop:free-loop} There is a weak equivalence of augmented $k$-algebras 
\[
\cal L A \, \, \simeq \,\, A_1 \, .
\]
\end{prop}

\begin{proof} This is claimed by Lazarev {\cite[th.~4.1]{Lazarev2}}, but
we were unable to understand his argument.  
Fortunately, we were helped out by Mike Mandell, who explained
a different proof to us.
We sketch Mandel's argument below when $k = S$ and leave
the general case as an exercise for the reader. 

The $S$-algebra $\cal LS := 
F(S^1_+,S)$ is just the Spanier-Whitehead dual
of $S^1_+$ in the category of symmetric spectra; it 
has the structure of a (commutative) $S$-algebra (cf.\ \cite{Cohen}). 

The evident pairing
\[
A \smsh_S {\cal L}S^1 \to \cal LA
\]
is a weak equivalence of $S$-algebras. It is therefore enough to show that
$\cal LS^1 \simeq S \vee S^{-1}$ as $S$-algebras, since then
\[
\cal LA \simeq A \smsh_S \cal LS^1  \simeq A \smsh_S (S \vee S^{-1}) \cong A_1\, .
\]
Let $R$ be an augmented $k$-algebra. 
The idea of the remainder of the proof is to study the forgetful map 
\[
S\text{-}\Alg/S(R,S\vee S^{-1}) \to S\text{-}\Mod/S(R,S\vee S^{-1})\, ,
\]
which is just the map
\begin{equation} \label{eqn:forgetful-map}
S\text{-}\Alg/S(R,S\vee S^{-1}) \to S\text{-}\Alg/S(TR,S\vee S^{-1})\, ,
\end{equation}
induced by the algebra homomorphism $TR \to R$.

Using Proposition \ref{prop:Lazarev}, Remark \ref{rem:extension} and Lemma
\ref{lem:desuspend},
there is a homotopy fiber sequence
\begin{equation} \label{eqn:fibration-one}
S\text{-}\Alg/S(R,S\vee S^{-1}) \to R^e\text{-}\Mod(R,S) \to \Omega^\infty S\, ,
\end{equation}
where the displayed homotopy fiber is taken at the basepoint of $\Omega^\infty S$
given by the unit. Note that the $R^e$-module structure on $S$ arises
by augmentation, so an extension by scalars argument shows 
\begin{equation} \label{eqn:scalars}
R^e\text{-}\Mod(R,S) \simeq R\text{-}\Mod(S,S) \, .
\end{equation}
Combining the fiber sequence \eqref{eqn:fibration-one} with
this last identification yields a homotopy fiber sequence
\begin{equation} \label{eqn:fibration-two} 
S\text{-}\Alg/S(R,S\vee S^{-1}) \to R\text{-}\Mod(S,S) \to \Omega^\infty S\, .
\end{equation}
With respect to the homomorphism $TR\to R$, we obtain a diagram
\begin{equation}\label{eqn:crux-diagram-one}
\xymatrix{
S\text{-}\Alg/S(R,S\vee S^{-1}) \ar[r] \ar[d] & R\text{-}\Mod(S,S) \ar[d] \\
S\text{-}\Mod/S(R,S\vee S^{-1}) \ar[r] & TR\text{-}\Mod(S,S)
}
\end{equation}
which is homotopy cartesian by \eqref{eqn:fibration-two} 
(where $R$ replaced by $TR$ in \eqref{eqn:fibration-two} for the
bottom horizontal map of \eqref{eqn:crux-diagram-one}). Henceforth, we specialize to the case $R =\cal LS$ (but the argument below works equally well for any $k$-algebra
which is weakly equivalent to $S\vee S^{-1}$ as an augmented
$S$-module; cf.\ Remark \ref{rem:example} below).

Note that
\[
\pi_0(S\text{-}\Mod/S(R,S\vee S^{-1})) \cong \pi_0(S\text{-}\Mod(R,S^{-1}))
\cong \Bbb Z\, ,
\] 
since, as augmented $S$-modules, $R \simeq S \vee S^{-1}$.
Furthermore, up to homotopy, such a weak equivalence
corresponds to one of the two possible generators of $\Bbb Z$. 
To lift either of these weak equivalences to an algebra map,
it suffices to show that the right vertical map of the diagram 
\eqref{eqn:crux-diagram-one} is surjective on $\pi_0$. In fact,
we will show that the right vertical map is a retraction up 
homotopy.

It is reasonably well-known that 
$R\text{-}\Mod(S,S)$, considered as an $S$-module, coincides up to homotopy
 with $S$-$\Mod(B_{\text{alg}}R,S)$ where $B_{\text{alg}}R$
is the bar construction on $R$ in the category of augmented $S$-algebras.
Similarly, one can show that  $TR\text{-}\Mod(S,S)$ coincides up to homotopy with
$S$-$\Mod(B_{\text{alg}}TR,S)$. We need to understand the map
\begin{equation} \label{eqn:bar-maps}
B_{\Alg}TR \to B_{\Alg}R\, .
\end{equation}
The bar construction $B_{\Alg}R$ is not hard to identify as an $S$-module. 
The homotopy spectral sequence defined by the skeletal filtration
has $E^2$-term $E^2_{p,q} = \pi_q(S^{-p})$. It is a spectral sequence
of $\pi_\ast(S)$-modules and it evidently
degenerates at the $E^2$-page. So
we obtain a weak equivalence of $S$-modules
\[
B_{\Alg}R \simeq \bigvee_{j \ge 0} S\, .
\]
This computation shows $B_{\Alg}R$ coincides with the associated graded of the filtration
defined by skeleta.

As for $B_{\Alg}TR$, it coincides with $B_{\Mod}R$, the bar
construction of $R$ considered as an augmented $S$-module with respect to the 
monoidal structure given by the coproduct of augmented modules. 
Furthermore, $B_{\Mod}R$
is easily identified with $\Sigma_S R$, the (fiberwise) suspension of $R$ considered
as an augmented $S$-module.
As $R \simeq S \vee S^{-1}$ as augmented $S$-modules, we have $\Sigma_S R \simeq S \vee S$.
Therefore \eqref{eqn:bar-maps} amounts to the map
\[
\bigvee_{0 \le j \le 1} S \to \bigvee_{j\ge 0} S 
\]
given by the inclusion of the $1$-skeleton into $B_{\Alg}R$.
It is clear that this inclusion is a split summand, so the restriction map 
$S\text{-}\Mod(B_{\Alg}R,S) \to S\text{-}\Mod(B_{\Alg}TR,S)$ is a retraction
up to homotopy. In particular,  the right vertical map of \eqref{eqn:crux-diagram-one}
is a surjection on $\pi_0$.
\end{proof}

\begin{rem} \label{rem:example} The above proof actually shows that any augmented
$S$-algebra $R$ equipped with a weak equivalence to $S \vee S^{-1}$ as
an augmented $S$-module has a lifting to a weak equivalence 
as an augmented $S$-algebra. Furthermore, the proof gives
a homotopy fiber sequence
\[
\prod_{j \ge 2} \Omega^\infty S \to S\text{-}\Alg/S(R,S\vee S^{-1}) \to 
S\text{-}\Mod/S(R,S\vee S^{-1})\, .
\]
A version of this sequence also holds in the unaugmented case.
\end{rem}

We apply Proposition \ref{prop:free-loop} in the following instance. By the adjunction property, we have
\begin{equation} \label{eqn:exponential}
\Omega_1 {k\text{-}\Alg}(R,A) \cong {k\text{-}\Alg/A}(R,\cal L A)\, .
\end{equation}
We are now in a position to deduce Theorem \ref{bigthm:crux}:

\begin{cor} \label{cor:crux} Let $1\: R\to A$ be a $k$-algebra homomorphism .
Then there is a natural weak equivalence
\[
\Omega_1 {k\text{-}\Alg}(R,A) \simeq D^\bullet(R;A)\, .
\]
\end{cor}

\begin{proof} This uses the chain of weak equivalences
\begin{alignat*}{2}
\Omega_1 {k\text{-}\Alg}(R,A) &\simeq   
{k\text{-}\Alg/A}(R,A \vee \Sigma^{-1} A) \, , &&\text{ by \eqref{eqn:exponential}  and Prop.\ \ref{prop:free-loop} }\\
& = \Der(R,\Sigma^{-1}A)\, , &&\text{ by definition } \\
& \simeq
\Omega \Der(R;A)\, , &&\text{ by Lem.\ \ref{lem:desuspend} } \\
& \simeq D^\bullet(R;A) 
&&\text{ by Cor.\ \ref{cor:reduced-hh-derivations}}. 
\end{alignat*}
\end{proof}

\section{Proof of Theorem \ref{bigthm:main}\label{sec:main}}

Let $R$ be a $k$-algebra.
By essentially the same argument appearing in \S\ref{sec:spaces}, there is a 
homotopy fiber sequence of categories
\begin{equation} \label{eqn:fibering-rings}
\cal C_R \to \coprod_{[R']} hk\text{-}\Alg(R') \to  hk\text{-}\Mod(R)  \, ,
\end{equation}
in which the decomposition appearing in the middle is indexed over those components
of $hk$-$\Alg$ (= the category of $k$-algebra weak equivalences), 
which have the property that $R'$ is weak equivalent to $R$ as
a $k$-module. In other words, the middle category is the 
full subcategory of $hk$-$\Alg$ whose objects are weak equivalent to 
$R$ as a $k$-module.  Similarly, $hk\text{-}\Mod(R)$ denotes component
the category of  weak equivalences which contains $R$. The categories 
appearing in \eqref{eqn:fibering-rings} have a preferred basepoint
determined by the $k$-algebra $R$ and
$\cal C_R$ corresponds to the homotopy fiber at the basepoint.

The contracting data $\cal U$ for \eqref{eqn:fibering-rings} is given by the category 
whose objects are pairs $(N,h)$ where $N$ is an $R$-module and
$h\: R \to N$ is a weak equivalence of $R$-modules. A morphism $(N,h) \to (N',h')$ is 
a map $f\: N \to N'$ such that $h' = f\circ h$. The functor 
$\cal C_R \to \cal U$ is the forgetful functor, as is the functor 
$\cal U \to hk\text{-}\Mod(R)$. Moreover, $\cal U$ is contractible, since $(R,\text{id})$ is an initial object.

As in \S\ref{sec:spaces}, the strategy will be to identify the 
middle and last terms of \eqref{eqn:fibering-rings} as double deloopings
of the Hochschild cohomology spaces.

\begin{lem} \label{lem:HH-delooped} There are weak equivalences of based spaces
\[
D^\bullet(R;R) \,\, \simeq\,\,  \Omega^2 |hk\text{-}\Alg(R)|
\]
and 
\[
D^\bullet(TR;R) \,\, \simeq \,\, \Omega^2 |hk\text{-}\Mod(R)|\, .
\]
\end{lem}

\begin{proof} Using the Dwyer-Kan equivalence,
 double loop space of $|hk\text{-}\Alg(R)|$ taken at the point defined by $R$ 
is identified with 
\[
\Omega^2 Bh\Aut_{k\text{-}\Alg}(R) \simeq \Omega_1 {k\text{-}\Alg}(R,R)\, .
\]
Using Corollary \ref{cor:crux} applied to the identity
map $R\to R$, we obtain a weak equivalence
\[
\Omega_1 {k\text{-}\Alg}(R,R)  \simeq D^\bullet(R;R) \, .
\] 
This establishes the first part of the lemma.

For the second part, we use the chain of identifications,
\begin{align*}
\Omega^2 |hk\text{-}\Mod(R)| \,\, & \simeq \,\,  \Omega^2 Bh\Aut_{k\text{-}\Mod}(R)\, ,\\
 \,\, & \simeq \,\, \Omega_1 h\Aut_{k\text{-}\Mod}(R) \, ,\\
  \,\, & \cong \,\,  \Omega_1 {k\text{-}\Alg}(TR,R)\, ,\\
  \,\, & \simeq \,\,D^\bullet(TR;R)\, .
\end{align*}
where the last weak equivalence is obtained from Corollary \ref{cor:crux} 
applied to the $k$-algebra map $TR \to R$.
\end{proof}

\begin{proof}[Proof of Theorem \ref{bigthm:main}]
Use
the homotopy fiber sequence \eqref{eqn:fibering-rings} together with 
Lemma \ref{lem:HH-delooped}. Note that with 
the  deloopings, the map
\[
D^\bullet(R;R) \to 
D^\bullet(TR;R) 
\]
has a preferred double delooping given by realizing the forgetful functor 
\[
\coprod_{[R']} hk\text{-}\Alg(R') \to  hk\text{-}\Mod(R)\, .   \qedhere
\]
\end{proof}

\section{The augmented case \label{sec:augment}}

\begin{defn} For an augmented $k$-algebra $R$ we define
the {\it moduli space of augmented $k$-algebra} structures on $R$,
\[
\cal M_{R/k} \, ,
\]
to be the classifying space of the category whose objects
are pairs $(E,h)$ in which $E$ is an augmented $k$-module and 
$h\: E \to R$ is a weak equivalence of augmented $k$-modules.
A morphism $(E,h) \to (E',h')$ is an augmentation preserving map
$f\: E \to E'$ such that $h'\circ f = h$.  
\end{defn}

It is a consequence of the definition that:
\begin{itemize}
\item There is an evident forgetful map 
\[
\cal M_{R/k}  \to \cal M_R\, .
\]
\item There is a homotopy fiber sequence
\[
\Omega^2 \cal M_{R/k} \to 
\Omega_1 k\text{-}\Alg/k(R,R) \to \Omega_1 k\text{-}\Mod/k(R,R)\, .
\]
\end{itemize}

\begin{defn} Let $M$ be an $R$-bimodule which is augmented over $k$. 
We set
\[
\HH^\bullet(R/k;M) := R^e\text{-}\Mod/k(R,M)\, ,
\]
i.e., the homotopy function complex associated to the augmented
bimodule maps $R \to M$.
\end{defn}

Given an augmented $k$-algebra map $1\: R \to A$, we may regard
$A$ is an augmented $R$-bimodule. Then restriction
defines a map of based spaces
\[
\HH^\bullet(R/k;A) \to \HH^\bullet(R^e/k;A)\, .
\]
and we let $D^\bullet(R/k;A)$ be the homotopy
fiber taken at the point defined by the bimodule map $R^e \to R\to A$.

The following is the augmented version of Theorem \ref{bigthm:main}. As the proof is
similar, we omit it. 

\begin{thm} \label{augmented} There is a homotopy fiber
sequence
\[
\cal M_{R/k} \to \cal B^2D^\bullet (R/k;A) \to \cal B^2
D^\bullet (TR/k;A)\, ,
\]
where the deloopings in each case are defined as in the proof of 
Theorem \ref{bigthm:main}.
\end{thm}

\section{Examples \label{sec:examples}}

Computations are somewhat easier to make in the augmented case, since in the unaugmented
setting one needs to understand $k\text{-}\Alg(R,k)$.

\subsection*{The square-zero case}
Let $R = S\vee S^{-1}$ be the trivial square zero extension of $S$.
We wish to study the homotopy type of $\cal M_R$ in this case.
The proof of  Proposition \ref{prop:free-loop}
 shows that the augmented version $\cal M_{R/S}$
is connected. Moreover, inspection of the proof
shows that it amounts to a computation of $\Omega\cal M_{R/S}$.
The result is that one gets a weak equivalence
\[
\Omega  \cal M_{R/S} \,\, \simeq\,\,  \prod_{j \ge 2} \Omega^\infty S\, .
\]
(cf.\ Remark \ref{rem:example}).
In the square-zero case, the relationship between $\Omega^2\cal M_{R/S}$
and $\Omega^2 \cal M_R$ is
easy to describe.

\begin{lem} \label{lem: moduli_as_L} When $R = S \vee S^{-1}$ is the trivial square zero extension,
there is a weak equivalence of based spaces
\[
\Omega^2 \cal M_R \,\, \simeq\,\,  L\Omega^2\cal M_{R/S} \, .
\]
\end{lem} 

\begin{proof}(Sketch).
Using the homotopy cartesian diagram of $S$-algebras
\[
\xymatrix{
R \ar[r] \ar[d] & S \ar[d] \\
S  \ar[r] & S \times S
}
\]
in which each map $S \to S\times S$ is the diagonal, we can 
apply the homotopy function complex out of $R$ to obtain a homotopy cartesian square
\[
\xymatrix{
S\text{-}\Alg(R,R) \ar[d] \ar[r] & S\text{-}\Alg(R,S) \ar[d] \\
S\text{-}\Alg(R,S) \ar[r] & S\text{-}\Alg(R,S)  \times S\text{-}\Alg(R,S) 
}
\]
which shows 
\[
S\text{-}\Alg(R,R) \simeq LS\text{-}\Alg(R,S)\, .
\]
On the other hand, there is a homotopy fiber sequence
\[
S\text{-}\Alg/S(R,R) \to S\text{-}\Alg(R,R) \to S\text{-}\Alg(R,S)
\]
and Proposition \ref{prop:free-loop} gives a weak equivalence
$S\text{-}\Alg/S(R,R) \simeq \Omega_1 S\text{-}\Alg(R,S)$.
A careful check of how the identification is made, which we omit, enables us to deduce
\[ 
S\text{-}\Alg(R,R)  \simeq LS\text{-}\Alg(R,S) \, .
\]

A similar argument in the module case shows
$S\text{-}\Mod(R,R) \simeq LS\text{-}\Mod(R,S)$.  Hence there are weak equivalences
\begin{align*}
\Omega^2 \cal M_R &\simeq L\text{hofiber}(\Omega_1 S\text{-}\Alg(R,S) \to \Omega_1 S\text{-}\Mod(R,S)) \\
& \simeq 
L\text{hofiber}(S\text{-}\Alg/S(R,R) \to S\text{-}\Mod/S(R,R))\\
& \simeq L\Omega^2\cal M_{R/S}\, .
\end{align*}
\end{proof}

\begin{cor} When $R = S\vee S^{-1}$ is the trival square-zero extension, there
is a weak equivalence
\[
\Omega^2 \cal M_R \,\, \simeq \,\, \prod_{j \ge 2} \Omega^\infty S^{-1} \times \Omega^\infty S^{-2}\, .
\]
\end{cor} 

\begin{proof} It was noted already that Remark \ref{rem:example}  gives the computation
$
\Omega \cal M_{R/S} \,\, \simeq\,\,  \prod_{j \ge 2} \Omega^\infty S 
$.
Take the based loops of both sides, then apply 
 free loops  and use Lemma \ref{lem: moduli_as_L}.
\end{proof}

\subsection*{Commutative group rings}
 Suppose $R = S[G] :=  S \smsh (G_+)$ is the group ring 
on a topological abelian group $G$ (for technical reasons,
we assume that the underlying space of $G$ is cofibrant). 
Then the adjoint action of $G$ acting on $R$ is trivial, and
it is not difficult exhibit  a  weak equivalence
\[
\HH^\bullet(R;R) \simeq F((B G)_+,\Omega^\infty R) \, ,
\]
where the space on the right is the function space of unbased maps
$BG \to \Omega^\infty R$.
Similarly,
\[
\HH^\bullet(TR;R) \simeq F((\Sigma G)_+,\Omega^\infty R)) \, .
\]
The map $\HH^\bullet(R;R) \to \HH^\bullet(TR;R)$ is induced in this case
by the inclusion $\Sigma G \to B G$ given by 
the 1-skeleton of $BG$. Let $X_G = BG/\Sigma G$. Using 
Corollary \ref{cor:main}, we obtain a weak equivalence
\[
\Omega^2 \cal M_R \,\, \simeq \,\, F(X_G,\Omega^\infty R)\, .
\]
In particular, $\pi_\ast(\cal M_R) = R^{2-\ast}(X_G)$
is the shifted $R$-cohomology of $X_G$
for $k \ge 2$.  The space $X_G$ comes equipped with a filtration, so one
gets an Atiyah-Hirzebruch spectral sequence converging to its $R$-cohomology.

Let us specialize to case when $G = \Bbb Z$ is the integers. Then we have
$R = S[\Bbb Z] = S[t,t^{-1}]$
is the Laurent ring over $S$ in one generator. In this instance 
\[
X_{\Bbb Z} \,\, \simeq \,\, 
\bigvee_{j} S^2
\]
is a countable infinite wedge of $2$-spheres and we infer
\[
\Omega^2 \cal M_{S[\Bbb Z]}  \,\, \simeq\,\,  \Omega^2 \prod_{j} \Omega^\infty S[\Bbb Z] \, .
\]
The right side admits a further decomposition into an countable infinite
product of copies of $\Omega^\infty S$, using the $S$-module identification 
$S[\Bbb Z] \simeq \vee_j S$.


\begin{thebibliography}{EKMM}
\bibliographystyle{invent}
\bibliography{john}

%\bibitem[A]{Angeltveit}%
% Angeltveit, V.:
% The cyclic bar construction on $A_\infty$ $H$-spaces. 
%\newblock{Adv. Math.} {\bf 222} (2009), 1589--1610

\bibitem[A]{Angeltveit}%
Angeltveit, V.:
Topological Hochschild homology and cohomology of $A_\infty$ ring spectra. 
\newblock{Geom. Topol.} {\bf 12} (2008), 987--1032.

%\bibitem[B1]{Becker}%
%Becker, J.C.: Cohomology and the classification of liftings.
%\newblock {\rm Trans. Amer. Math. Soc.\rm} {\bf 133}, 447--475 (1968)

\bibitem[BV]{Boardman-Vogt}%
Boardman, J.~M.; Vogt, R.~M.: Homotopy invariant algebraic structures on topological spaces. \newblock Lecture Notes in Mathematics, Vol. 347. Springer-Verlag, Berlin-New York, 1973.

\bibitem[C]{Cohen}%
Cohen, R.~L.: Multiplicative properties of Atiyah duality. 
\newblock {\it Homology Homotopy Appl.} {\bf 6} (2004), 269--281.

\bibitem[DS]{Dugger-Shipley}%
Dugger, D., Shipley B.: Postnikov extensions of ring spectra.
\newblock {\it Alg. \& Geom. Topol} {\bf 6} (2006) 1785--1829.

\bibitem[D]{Dwyer}%
Dwyer, W.~G.: {\it Private communication.}
\newblock

\bibitem[DGI]{DGI}%
Dwyer, W.~G., Greenlees, J.~P.~C., Iyengar, S.: Duality in algebra and topology. 
\newblock {\it Adv. Math.} {\bf 200} (2006), 357--402.

\bibitem[DH]{DH}%
Dwyer, W., Hess, K.:
Long knots and maps between operads. 
\newblock{Geom. Topol.} {\bf 16} (2012), 919--955.


\bibitem[DK1]{DK1}%
 Dwyer W.~G., Kan D.~M.: 
Function complexes in homotopical algebra.
\newblock {\it Topology} {\bf 19}
(1980), 427--440.

\bibitem[DK2]{DK2}%
 Dwyer W.~G., Kan D.~M.: 
A classiÞcation theorem for diagrams of simplicial sets.
\newblock {\it Topology} {\bf 23} (1984), 139--155.




\bibitem[EKMM]{EKMM}%
 Elmendorf, A.~D., Kriz, I., Mandell, M.~A., May, J.~P. 
(with an appendix by Cole, M.), {\it Rings, modules, and algebras in stable homotopy theory}, \newblock Mathematical Surveys and Monographs 47 (American Mathematical Society, Providence, RI, 1997).

\bibitem[F]{Fuchs}%
Fuchs, M.:
Verallgemeinerte Homotopie-Homomorphismen und klassifizierende R\"aume.
\newblock{\it Math. Ann.} {\bf 161} (1965), 197--230.

\bibitem[KSS]{KSS}%
 Klein, J.~R., Schochet, C.~L., Smith, S.B.:
Continuous trace $C^\ast$-algebras, gauge groups and rationalization,
\newblock{\it J. Topol. Anal.} {\bf 1} (2009), 261--288.


\bibitem[La1]{Lazarev1}%
Lazarev, A.: Homotopy theory of $A_\infty$ ring spectra and applications to $M\text{\rm U}$-modules. 
\newblock {\it $K$-Theory} {\bf 24} (2001), 243Ð-281.

\bibitem[La2]{Lazarev2}%
Lazarev, A.: Spaces of multiplicative maps between highly structured ring spectra. 
\newblock {\it Categorical decomposition techniques in algebraic topology (Isle of Skye, 2001)}, 
237--259, Progr. Math., 215, BirkhŠuser, Basel, 2004.

\bibitem[La3]{Lazarev3}%
Lazarev, A.: 
Cohomology theories for highly structured ring spectra.
\newblock {\it Structured ring spectra}, 201--231, London Math. Soc. Lecture Note Ser., 315, Cambridge Univ. Press, Cambridge, 2004.

\bibitem[Lu]{Lurie}%
Lurie, J.: Higher Algebra.\\
\newblock  http:/\!\!/www.math.harvard.edu/$\sim$lurie/papers/HigherAlgebra.pdf




\bibitem[M]{May_fibration}%
May, J.~.P: Classifying spaces and fibrations. 
\newblock{\it Mem. Amer. Math. Soc.} {\bf 1} (1975), 1, no.\ 155.

\bibitem[Q]{Quillen}%
Quillen, D.~G.
Homotopical algebra. 
\newblock Lecture Notes in Mathematics, No. 43, Springer-Verlag, Berlin-New York, 1967


\bibitem[R1]{Robinson1}%
Robinson, A.:
Obstruction theory and the strict associativity of Morava K-theories. 
{\it Advances in homotopy theory (Cortona, 1988)}, 143--152, 
London Math. Soc. Lecture Note Ser., 139, Cambridge Univ. Press, Cambridge, 1989.

\bibitem[R2]{Robinson2}%
Robinson, A.: Classical obstructions and S-algebras. 
\newblock {\it Structured ring spectra}, 133--149, 
London Math. Soc. Lecture Note Ser., 315, Cambridge Univ. Press, Cambridge, 2004.

\bibitem[SV]{Schwaenzl-Vogt}%
Schw\"anzl, R., Vogt, R.~M.:
The categories of $A_\infty$- and $E_\infty$--monoids and ring spaces as closed simplicial and topological model categories. 
\newblock{\it Arch. Math. (Basel)} {\bf 56} (1991), 405--411

\bibitem[SS]{Schwede-Shipley}%
Schwede, S., Shipley, B.~E.:
Algebras and modules in monoidal model categories. 
\newblock{\it Proc. London Math. Soc.} {\bf 80} (2000), 491Ð-511.

\bibitem[St1]{Stasheff1}%
Stasheff, J.~D.: Homotopy associativity of H-spaces. I, II. 
\newblock{\it Trans. Amer. Math. Soc.} {\bf 108} (1963), 275--292; ibid. {\bf 108} 1963 293--312.

\bibitem[St2]{Stasheff2}%
Stasheff, J.: H-spaces from a homotopy point of view. 
\newblock Lecture Notes in Mathematics, Vol. 161 Springer-Verlag, Berlin-New York 1970.






\end{thebibliography}
\end{document}